\documentclass[nosumlimits,twoside]{amsart}
\usepackage{amsfonts, amsmath, amssymb}
\usepackage{graphicx}
\usepackage{float}
\usepackage{srcltx}
\usepackage[all]{xy}
\usepackage{enumerate}

\newtheorem{theorem}{\sc Theorem}[section]
\newtheorem{proposition}[theorem]{\sc Proposition}
\newtheorem{notation}[theorem]{\sc Notation}

\newtheorem{lemma}[theorem]{\sc Lemma}

\theoremstyle{definition}
\newtheorem{definition}[theorem]{\sc Definition}

\newtheorem{example}[theorem]{\sc Example}

\theoremstyle{remark}
\newtheorem{remark}[theorem]{\sc Remark}
\newtheorem{remarks}[theorem]{\sc Remarks}

\usepackage{cancel}
\usepackage[normalem]{ulem}

\begin{document}

\title{Monadic Decompositions and Classical Lie Theory}
\author{Alessandro Ardizzoni}
\address{University of Turin, Department of Mathematics "G. Peano", via Carlo Alberto 10, I-10123 Torino, Italy}  \email{alessandro.ardizzoni@unito.it}
\urladdr{www.unito.it/persone/alessandro.ardizzoni}

\author{Jos\'{e} G\'{o}mez-Torrecillas}
\address{Universidad de Granada, Departamento de \'{A}lgebra. Facultad de Ciencias, E-18071-Granada, Spain}
\email{gomezj@ugr.es}
\urladdr{http://www.ugr.es/local/gomezj}

\author{Claudia Menini}
\address{University of Ferrara, Department of Mathematics, Via Machiavelli
35, Ferrara, I-44121, Italy} \email{men@unife.it}
\urladdr{http://www.unife.it/utenti/claudia.menini}

\subjclass[2010]{Primary 18C15; Secondary 16S30}
\thanks{This paper was written while the first and the last authors were members of
GNSAGA. The second author thanks the members of the Dipartimento di Matematica at the University of Ferrara for their warm hospitality during his visit in February 2010, when the work on this paper was initiated. Research partially supported by the grant MTM2010-20940-C02-01 from the MICINN and FEDER. The authors are grateful to the referee for several useful comments,  and to Bachuki Mesablishvili for pointing to us the bibliography on monadic decompositions.}

\begin{abstract}
We show that the functor from bialgebras to vector spaces sending a bialgebra to its subspace of primitives has monadic length at most $2$.
\end{abstract}

\keywords{Monads, Lie algebras}
\maketitle

\section*{Introduction}
Given a functor $R_0 : \mathcal{A} \to \mathcal{B}_0$ with left adjoint $L_0 : \mathcal{B}_{0} \to \mathcal{A}$ we get, following \cite{AHW,MS}, and under suitable hypotheses, a sequence of adjoint pairs of functors

\begin{equation*}
\xymatrix@C=2cm{
\mathcal{A}\ar@<.5ex>[d]^{R_0}&\mathcal{A}\ar@<.5ex>[d]^{R_1}\ar[l]_{\mathrm{Id}_{\mathcal{A}}}&\mathcal{A}\ar@<.5ex>[d]^{R_2}\ar[l]_{\mathrm{Id}_{\mathcal{A}}}&\dots \ar[l]_{\mathrm{Id}_{\mathcal{A}}}\\
\mathcal{B}_0\ar@<.5ex>@{.>}[u]^{L_0}&\mathcal{B}_1\ar@<.5ex>@{.>}[u]^{L_1} \ar[l]_{U_{0,1}}&\mathcal{B}_2 \ar@<.5ex>@{.>}[u]^{L_2} \ar[l]_{U_{1,2}}&\dots\ar[l]_{U_{2,3}}
  }
  \end{equation*}
where for $i \geq 0$, $\mathcal{B}_{i+1}$ is the Eilenberg-Moore category of the monad $(L_i,R_i)$, $R_{i+1}$ is the comparison functor, and $U_{i,i+1}$ is the corresponding forgetful functor.  It is natural to inquire wether this process stops, as was done in \cite{AHW,MS}. To be more specific, the monadic length of $R_0$ is the first $N$ such that $U_{N,N+1}$ is an isomorphism of categories. In many basic examples, the functor $R_0$ is monadic and, therefore, it has monadic length  at most $1$. In this note, we show that the functor $P$ from  bialgebras to vector spaces sending a bialgebra to its subspace of primitives has monadic length $2$ (Theorem \ref{primitives}).

Section \ref{monadic} contains some remarks on the monadic decompositions of functors studied in \cite{AHW, MS} and their relationship with idempotent monads (\cite{AT}). The basic case of the adjoint pair encoded by a bimodule over unital rings is described in Remark \ref{bimodules}, with an eye on descent theory for modules. We also study the existence of comonadic decompositions under separability conditions (Proposition \ref{separable}).

Section \ref{examples} contains the aforementioned monadic decomposition of monadic length at most $2$ of the functor $P$ from bialgebras to vector spaces.

\section{Monadic decompositions}\label{monadic}

Consider categories $\mathcal{A}$ and $\mathcal{B}$.  Let $\left( L:\mathcal{B}\rightarrow \mathcal{A},R:\mathcal{A}\rightarrow
\mathcal{B}\right) $ be an adjunction with unit $\eta $ and counit $\epsilon
$, and consider the monad $\left( RL,R\epsilon L,\eta \right)$ generated on $\mathcal{B}$. By $\mathcal{B}_1$ we denote its Eilenberg-Moore category of algebras.  
Hence we can consider the so-called \emph{comparison functor} of the
adjunction $\left( L,R\right) $ i.e. the functor%
\begin{equation*}
K :\mathcal{A}\rightarrow \mathcal{B}_1,{\mathcal{\qquad }}%
KX:=\left( RX,R\epsilon X\right) ,{\mathcal{\qquad }}Kf:=Rf.
\end{equation*}%

Recall that the functor $R:\mathcal{A}%
\rightarrow \mathcal{B}$ is called \emph{monadic} (tripleable in Beck's
terminology \cite[Definition 3', page 8]{Beck}) whenever the comparison functor
$K:\mathcal{A}\rightarrow \mathcal{B}_1$ is an equivalence of
categories.

\subsection{Idempotent monads and monadic decompositions}
The notion of an idempotent monad is, as we will see below, tightly connected with the monadic length of a functor.

\begin{definition}
\cite[page 231]{AT} A monad $\left( Q,m,u\right) $ with multiplication $m$ and unit $u$ is called \emph{idempotent%
} whenever $m$ is an isomorphism. An adjunction $\left( L,R\right) $ is
called \emph{idempotent} whenever the associated monad is idempotent.
\end{definition}

There are several basic characterizations of idempotent adjunctions (see \cite{AT,MS}). In particular, idempotency of an adjunction means equivalently that any of the four factors appearing in the two triangular identities of the adjoint pair $(L,R)$ is an isomorphism (\cite[Proposition 2.8]{MS}).\medskip

In the following we will denote by $U : \mathcal{B}_1 \to \mathcal{B}$ the forgetful functor and by $F : \mathcal{B} \to \mathcal{B}_1$ the free functor associated to an adjunction $(L,R)$.

\begin{remark}\label{re:etaU}
  Note that the adjunctions $(F,U)$ and $(L,R)$ have the same associated monad so that $\left( L,R\right) $ is idempotent if and only if $\left( F,U\right) $ is  if and only if, by \cite[Proposition 2.8]{MS}, one has $\eta U$ is an isomorphism.
\end{remark}

\begin{proposition}\label{pro:idempotent} For and adjunction $\left( L:\mathcal{B}\rightarrow \mathcal{A},R:%
\mathcal{A}\rightarrow \mathcal{B}\right) $ with unit $\eta $
and counit $\epsilon$, the
following assertions are equivalent.

\begin{enumerate}[(a)]

\item\label{idempotente} $\left( L,R\right) $ is idempotent,

\item\label{EMisomorfismo} the structure map of every object in $\mathcal{B}_1$ is an isomorphism,

\item\label{unidad}  $LU$ is a left adjoint of the comparison functor $K:%
\mathcal{A}\rightarrow \mathcal{B}_1$ of $\left( L,R\right) $, and  $\eta U=U\eta _{1}$, where $\eta _{1}$ is the unit of the new adjunction (being the new counit still $\epsilon$).

\end{enumerate}
 Moreover, if one of these conditions holds, then $LU$ is full and faithful.
\end{proposition}

\begin{proof}
\eqref{idempotente} $\Rightarrow$ \eqref{EMisomorfismo}.  By \cite[Proposition 2.8]{MS}, we have that $\eta R L = R L \eta$. Now, for any algebra $(X, \mu) \in \mathcal{B}_1$, we have that $\mu \circ \eta X = 1_X$. Moreover, by naturality of $\eta$, we know that $\eta X \circ \mu = R L \mu\circ\eta R L X=R L \mu\circ R L \eta X=1_{RLX}$. Therefore, $\mu$ is an isomorphism.

\eqref{EMisomorfismo} $\Rightarrow$ \eqref{unidad}. If $(X, \mu)$ is an algebra over $RL$, then $\mu : RLX \to X$ is an homomorphism of $RL$--algebras. Now, if $\mu$ is an isomorphism, then, necessarily, $\mu = (\eta X)^{-1}$. We get easily that $\eta X$ is a morphism in $\mathcal{B}_1$. Therefore, $\eta$ will serve as the unit for an adjunction $(LU,K)$ (being the counit $\epsilon$).

\eqref{unidad} $\Rightarrow$ \eqref{idempotente}. Given an object $Y$ of $\mathcal{A}$, consider its free $RL$--algebra  $(RY, R \epsilon Y)$. Since, by assumption, $\eta RY$ is a homomorphism of $RL$--algebras, we get the identity $R \epsilon LR Y \circ R L \eta RY = \eta R Y \circ R \epsilon Y$, which implies, by the triangular identities for the adjunction $(L,R)$, that $\eta R Y \circ R \epsilon Y = 1_{RLRY}$, whence $R \epsilon Y$ is an isomorphism. With $Y = LX$ for any object $X$ of $\mathcal{B}$, one obtains that $R \epsilon L X$ is an isomorphism. Hence, $(L,R)$ is idempotent.

Let us prove the last part of the statement. By Remark \ref{re:etaU}, condition \eqref{idempotente} is equivalent to $\eta U$ isomorphism.  From $\eta U=U\eta _{1}$ and the fact that $U$ reflects isomorphisms we deduce that $\eta _{1}$ is an isomorphism so that $LU $ is full and
faithful.
\end{proof}

\begin{remark}\label{rem:ff}
 By \cite[Proposition 2.8]{MS}, if $L$ is full and faithful, then the adjunction $(L,R)$ is idempotent. On the other hand, since the units of the adjunctions $(L,R)$ and $(F,U)$ are equal, we get from Proposition \ref{pro:idempotent} that $L$ is full and faithful if and only if $U$ is an isomorphism of categories if and only if $U$ is an equivalence of categories.
\end{remark}

Before defining monadic decompositions and monadic length, we derive the following consequence of Proposition \ref{pro:idempotent}, which is interpreted in the classical setting of epimorphisms of rings in Example \ref{ex:GT}.

\begin{proposition}
\label{pro:Rff}Let $\left( L,R\right) $ be an adjunction. The following
are equivalent.

\begin{itemize}
\item[$(1)$] $R$ is full and faithful.

\item[$\left( 2\right) $] $\left( L,R\right) $ is idempotent and $R$ is
monadic.
\end{itemize}
\end{proposition}

\begin{proof} Let $\eta $ be the unit and $\epsilon $ be the counit of the adjunction $%
\left( L,R\right) $.

$\left( 1\right) \Rightarrow \left( 2\right) $. By assumption $\epsilon $ is an isomorphism so that $%
R\epsilon L$ is an isomorphism i.e. $\left( L,R\right) $ is idempotent.
By Proposition \ref{pro:idempotent}, we know that $\Lambda :=LU$ is a
left adjoint of the comparison functor $K:\mathcal{A}\rightarrow {%
\mathcal{B}_1}$ of $\left( L,R\right) $. Moreover $\Lambda $ is full and
faithful and $\epsilon A=\epsilon _{1}A$ where $\epsilon _{1}$ is the counit
of the adjunction $\left( \Lambda ,K\right) .$ Thus $\epsilon _{1}$ is an
isomorphism i.e. $K$ is full and faithful too. Hence $K$ is an
equivalence.

$\left( 2\right) \Rightarrow \left( 1\right) $. Since $\left( L,R\right) $
is idempotent, by Proposition \ref{pro:idempotent}, we get $\epsilon
A=\epsilon _{1}A.$ Since $R$ is monadic, the comparison functor is an equivalence and hence $\epsilon
_{1}A$ is an isomorphism. Hence $\epsilon A$ is an isomorphism so that $R$
is full and faithful.
\end{proof}

\begin{example}\label{ex:GT}
Let us consider a morphism of rings $\pi :B\rightarrow A$, and its
canonical associated adjunction
\begin{equation*}
L:\mathrm{Mod}\text{-}B\rightarrow \mathrm{Mod}\text{-}A,\qquad R:\mathrm{Mod%
}\text{-}A\rightarrow \mathrm{Mod}\text{-}B.
\end{equation*}%
By \cite[Proposition 1.2, page 226]{Stenstroem}, $\pi $ is an epimorphism if
and only if the counit of the adjunction is an isomorphism. This is
equivalent to say that $R$ is full and faithful. Thus, by Proposition \ref%
{pro:Rff}, $\pi $ is an epimorphism if and only if $\left( L,R\right) $
is idempotent and $R$ is monadic. Note that $L$ needs not to be full and faithful.
Thus, when $\pi $ is an epimorphism, since $R$ is monadic, it has a monadic
decomposition of monadic length $1$ in the sense of Definition \ref{def:
comparable} but a monadic decomposition of (essential) length $0$ in the
sense of \cite[Definition 2.1]{AHW}.
\end{example}

\begin{definition}
\label{def: comparable} (See \cite[Definition 2.1]{AHW} and \cite[%
Definitions 2.10 and 2.14]{MS}) Fix a $N\in \mathbb{N}$. We say that a functor $R$ has a\emph{\ monadic decomposition of monadic
length }$N$\emph{\ }whenever there exists a sequence $\left( R_{n}\right)
_{n\leq N}$ of functors $R_{n}$ such that

1) $R_{0}=R$;

2) for $0\leq n\leq N$, the functor $R_{n}$ has a left adjoint functor $%
L_{n} $;

3) for $0\leq n\leq N-1$, the functor $R_{n+1}$ is the comparison functor
induced by the adjunction $\left( L_{n},R_{n}\right) $ with respect to its
associated monad;

4) $L_{N}$ is full and faithful while $L_{n}$ is not full and faithful for $%
0\leq n\leq N-1.$

Compare with the construction performed in \cite[1.5.5, page 49]{Manes-PhD}.

Note that for functor $R:\mathcal{A}\rightarrow \mathcal{B}$ having a
monadic decomposition of monadic length\emph{\ }$N$, we have a diagram

\begin{equation}
\xymatrix@C=2cm{
\mathcal{A}\ar@<.5ex>[d]^{R_0}&\mathcal{A}\ar@<.5ex>[d]^{R_1}\ar[l]_{\mathrm{Id}_{\mathcal{A}}}&\mathcal{A}\ar@<.5ex>[d]^{R_2}\ar[l]_{\mathrm{Id}_{\mathcal{A}}}&\cdots \ar[l]_{\mathrm{Id}_{\mathcal{A}}}\quad\cdots&\mathcal{A}\ar@<.5ex>[d]^{R_N}\ar[l]_{\mathrm{Id}_{\mathcal{A}}}\\
\mathcal{B}_0\ar@<.5ex>@{.>}[u]^{L_0}&\mathcal{B}_1\ar@<.5ex>@{.>}[u]^{L_1} \ar[l]_{U_{0,1}}&\mathcal{B}_2 \ar@<.5ex>@{.>}[u]^{L_2} \ar[l]_{U_{1,2}}&\cdots\ar[l]_{U_{2,3}}\quad\cdots&\mathcal{B}_N\ar@<.5ex>@{.>}[u]^{L_N}\ar[l]_{U_{N-1,N}}
  }\label{diag:MonadicDec}\end{equation}
where $\mathcal{B}_{0}=\mathcal{B}$ and, for $1\leq n\leq N,$

\begin{itemize}
\item $\mathcal{B}_{n}$ is the category of $\left( R_{n-1}L_{n-1}\right) $%
-modules ${_{R_{n-1}L_{n-1}}\mathcal{B}}_{n-1}$;

\item $U_{n-1,n}:\mathcal{B}_{n}\rightarrow \mathcal{B}_{n-1}$ is the
forgetful functor ${_{R_{n-1}L_{n-1}}}U$.
\end{itemize}

We will denote by $\eta _{n}:\mathrm{Id}_{\mathcal{B}_{n}}\rightarrow
R_{n}L_{n}$ and $\epsilon _{n}:L_{n}R_{n}\rightarrow \mathrm{Id}_{\mathcal{A}%
}$ the unit and counit of the adjunction $\left( L_{n},R_{n}\right) $
respectively for $0\leq n\leq N$. Note that one can introduce the forgetful
functor $U_{m,n}:\mathcal{B}_{n}\rightarrow \mathcal{B}_{m}$ for all $m\leq n
$ with $0\leq m,n\leq N$.
\end{definition}

\begin{remarks}
1) Assume that $R$ fits into a diagram such as \eqref{diag:MonadicDec}. If $R_{N-1}$ is monadic i.e. $R_N$ is a category equivalence, then obviously $L_N$ is full and faithful so that $R_0$ has a monadic decomposition of monadic length at most $N$. Nevertheless if $R_0$ has monadic length $N$, then $R_{N}$ needs not to be an equivalence.

2) The notion of \emph{comonadic decomposition of comonadic length }$N$ can be
easily introduced. In this case we will use the notations $\left(
L^{n},R^{n}\right) _{n\in \mathbb{N}}$ with superscripts and require that $R^{N}$ is full and faithful.
\end{remarks}

\begin{proposition}
\label{pro:idemmonad2}Let $\left( L:\mathcal{B}\rightarrow \mathcal{A},R:%
\mathcal{A}\rightarrow \mathcal{B}\right) $ be an idempotent adjunction.
Then $R:\mathcal{A}\rightarrow \mathcal{B}$ has a\emph{\ }monadic
decomposition of monadic length at most $1$.
\end{proposition}

\begin{proof}
By Proposition \ref{pro:idempotent}, $L_{1}=L_{0}U_{0,1}$ is full and
faithful.
\end{proof}

\begin{remark}
It follows from Remark \ref{rem:ff} that condition $4)$ in Definition \ref{def: comparable} is equivalent to the requirement that the forgetful functor
$U_{N,N+1} :\mathcal{B}_{N + 1} \to \mathcal{B}_N$ is an isomorphism of categories.
Thus, if  $R:\mathcal{A}\rightarrow \mathcal{B}$ has a\emph{\ }monadic
decomposition of monadic length $N\in\mathbb{N}$, then we can consider the comparison functor $R_{N+1}:\mathcal{A}%
\rightarrow \mathcal{B}_{N+1}$ of $\left( L_{N},R_{N}\right) .$
Moreover, still by Remark \ref{rem:ff}, $L_{N}$ full and faithful  implies that the adjunction $\left(
L_{N},R_{N}\right) $ is idempotent. Hence, by Proposition \ref{pro:idempotent}, $L_{N+1}:=L_{N}U_{N,N+1}$ is a left adjoint of $R_{N+1}$ (and $%
L_{N+1}$ is full and faithful too). Note that the fact that $R_{N+1}$ is a
right adjoint is assumed from the very beginning in \cite[Definition 2.1]{AHW}%
. By Proposition \ref{pro:idempotent} again, we deduce that $\eta
_{N}U_{N,N+1}=U_{N,N+1}\eta _{N+1}$ and $\epsilon _{N}A=\epsilon _{N+1}A$
where $\eta _{n}$ is the unit and $\epsilon _{n}$ is the counit of the
adjunction $\left( L_{n},R_{n}\right) $ for all $n\leq N+1.$ Iterating this
process we get that for all $M\geq N,$ the tower in (\ref{diag:MonadicDec})
can be extended with adjoints $\left( L_{M},R_{M}\right) $ where $L_{M}$ is
full and faithful so that $U_{M,M+1}:\mathcal{B}_{M+1}\rightarrow \mathcal{B}%
_{M}$ is a category isomorphism. Moreover $\eta _{M}U_{M,M+1}=U_{M,M+1}\eta
_{M+1}$ and $\epsilon _{M}A=\epsilon _{M+1}A.$ By the foregoing we have that%
\begin{equation*}
R=R_{0}=U_{0,1}\circ U_{1,2}\circ \cdots \circ U_{N-1,N}\circ R_{N}
\end{equation*}%
where $U_{0,1},U_{1,2},\cdots ,U_{N-1,N}$ are monadic functors but not
category isomorphisms. Moreover this is a maximal decomposition of this
form. This is essentially \cite[Remarks 2.2]{AHW}.
\end{remark}

\begin{remark}\label{coalgebras}
If $R : \mathcal{A} \to \mathcal{B}$ has a monadic decomposition of length $N$, then, since $L_N : \mathcal{B}_N \to \mathcal{A}$ is full and faithful, the dual of Proposition \ref{pro:Rff} gives that $L_N$ is a comonadic functor and $(L_N, R_N)$ is coidempotent. Thus, the comparison functor $C: \mathcal{B}_N \to \mathcal{A}^{N}$, where $\mathcal{A}^N$ denotes the category of $L_NR_N$--coalgebras, is an equivalence of categories.
\end{remark}

\subsection{Essentially surjective}

The following result determines the objects which are images of right
adjoint functors under suitable assumptions. This can be regarded as a sort
of descent theory for these functors.

\begin{notation}
Let $R:\mathcal{A}\rightarrow \mathcal{B}$. We will denote by $\mathrm{Im}R$
the full subcategory of $\mathcal{B}$ consisting of those objects $B\in
\mathcal{B}$ such that there is an object $A\in \mathcal{A}$ and an
isomorphism $B\cong RA$ in $\mathcal{B}$.

Recall that a functor $R:\mathcal{A}\rightarrow \mathcal{B}$ is \emph{%
essentially surjective} if $\mathrm{Im}R=\mathcal{B}$.
\end{notation}

\begin{proposition}
\label{teo: descentFunct}Suppose that $R:\mathcal{A}\rightarrow \mathcal{B}$
has a\emph{\ }monadic decomposition of monadic length $N\in
\mathbb{N}$. Let $n\in \left\{ 0,\ldots ,N\right\} $. Then

1) $\mathrm{Im}R\subseteq \mathrm{Im}U_{0,n}.$

2) $\mathrm{Im}R=\mathrm{Im}U_{0,n}$ whenever $R_{n}$ is essentially
surjective.

3) $\mathrm{Im}R=\mathrm{Im}U_{0,N}$.
\end{proposition}

\begin{proof}
It follows from the equalities $ U_{0,n}R_{n}=R_{0}=R.  \label{form:UR}$
\end{proof}

\begin{remark}\label{bimodules}
\label{re:descent}Proposition \ref{teo: descentFunct} can be considered as a
``general dual descent theory" result. In fact the theorem states that the
objects of $\mathcal{B}=\mathcal{B}_{0}$ which are isomorphic to objects of
the form $RA$, for some $A\in \mathcal{A}$, are exactly those of the form $%
U_{0,N}B_{N}$ where $B_{N}\in \mathcal{B}_{N}$. In particular, when $N=1$,
i.e. $L_{1}$ is full and faithful, we have that the objects of $\mathcal{B}$
which are isomorphic to objects of the form $RA$, for some $A\in \mathcal{A}$%
, are exactly those of the form $U_{0,1}B_{1}$ where $B_{1}\in \mathcal{B}%
_{1}$. This is exactly the dual form of classical descent
theory for (bi)modules. In fact, let $S,$ $T$ be rings and let ${_{S}M_{T}}$ be
a bimodule. Consider the following adjunction
\begin{equation*}
\begin{tabular}{lll}
$L:{\mathcal{M}_{S}}\rightarrow {\mathcal{M}_{T},}$ &  & $R:{\mathcal{M}_{T}}%
\rightarrow {\mathcal{M}_{S}}$ \\
$LX=X\otimes _{S}M{,}$ &  & $RY={\mathrm{Hom}_{T}}\left( M,Y\right) {,}$%
\end{tabular}%
\end{equation*}%
between the category ${\mathcal{M}_{S}}$ of right $S$-modules and the
category ${\mathcal{M}_{T}}$ of right $T$-modules. The category ${\mathcal{M}%
_{T}}$ has (co)equalizers. By (dual) Beck's Theorem \cite[Proof of Theorem 1]%
{Beck}, the comparison functors $R_{1}$ and $L^{1}$ have a left adjoint $%
L_{1}$ and a right adjoint $R^{1}$ respectively.

Assume that ${M}$ is flat as a left $S$-module. Then $L=L^{0}$ is exact so
that, the dual of Beck's Theorem ensures that $R^{1}$ is full and faithful. Therefore, $L$ admits a comonadic decomposition of comonadic length at most $1$.
Thus we have that the objects of ${\mathcal{M}_{T}}$ which are isomorphic to
objects of the form $LX$, for some $X\in {\mathcal{M}_{S}}$, are exactly
those of the form $U^{0,1}X^{1}$ where $X^{1}$ is an object of the category of $LR$--coalgebras  $\left( {\mathcal{M}%
_{T}}\right)^1$. Hence the category $\left( {\mathcal{M}_{T}}\right)^1$
solves the descent problem for modules. When $M_T$ is finitely generated and projective, we have an isomorphism of comonads $LR \cong - \otimes_T M^* \otimes_S M$ where $M^* \otimes_S M$ is the comatrix coring associated to ${}_S M_T$ (see \cite{ElKaoutit/Gomez:2003}, and \cite{Gomez:2006,Gomez/Vercruysse:2007} for more general bimodules). Coalgebras over $LR$ are precisely the comodules over the $T$--coring $M^* \otimes_S M$.

Assume that ${M}$ is projective as a right $T$-module. Then $R=R_{0}$ is
exact so that, Beck's Theorem ensures that $L_{1}$ is full and faithful, and $R$ has a comonadic decomposition of length at most $1$.
\end{remark}

\subsection{Separability}

Let $\left( Q,m,u\right) $ be a monad on a category $\mathcal{B}$, with multiplication $m$ and unit $u$. A \emph{right
module functor} on $\left( Q,m,u\right) $ is a pair $\left( W,\mu \right) $
where $W:\mathcal{B}\rightarrow \mathcal{A}$ is a functor and $\mu
:WQ\rightarrow W$ is a natural transformation such that
\begin{equation*}
\mu \circ \mu Q=\mu \circ Wm\qquad \text{and}\qquad \mu \circ Wu=\mathrm{Id}%
_{Q}.
\end{equation*}%
A morphism $f:\left( W,\mu \right) \rightarrow \left( W^{\prime },\mu
^{\prime }\right) $ of right module functors is a natural transformation $%
f:W\rightarrow W^{\prime }$ such that $\mu ^{\prime }\circ fQ=f\circ \mu .$

It is clear that $\left( WQ,Wm\right) $ is a right module functor on $\left(
Q,m,u\right) $ and that $\mu :\left( WQ,Wm\right) \rightarrow \left( W,\mu
\right) $ is morphism of right module functors. We will say that $\left(
W,\mu \right) $ is \emph{relatively projective} whenever $\mu :\left(
WQ,Wm\right) \rightarrow \left( W,\mu \right) $ splits as a morphism of
right module functors. Explicitly this means that there is a morphism $%
\gamma :\left( W,\mu \right) \rightarrow \left( WQ,Wm\right) $ of right
module functors such that $\mu \circ \gamma =\mathrm{Id}_{\left( W,\mu
\right) }$ i.e. that there is a natural transformation $\gamma :W\rightarrow
WQ$ such that $\mu \circ \gamma =\mathrm{Id}_{W}$ and $Wm\circ \gamma
Q=\gamma \circ \mu $.

Let $\left( L:\mathcal{B}\rightarrow \mathcal{A},R:\mathcal{A}\rightarrow
\mathcal{B}\right) $ be an adjunction with unit $\eta $ and counit $\epsilon
$. Then $\left( L,\epsilon L\right) $ is a right module functor on $\left(
RL,R\epsilon L,\eta \right) .$ In fact $\epsilon L\circ \epsilon
LRL=\epsilon L\circ LR\epsilon L$ and $\epsilon L\circ L\eta =\mathrm{Id}%
_{RL}.$

The notion of a separable functor was introduced in \cite{Nastasescu/alt:1989}. This concept is motivated by various examples, being perhaps the most fundamental the following. Given a homomorphism of rings $R \to S$, then the restriction of scalars functor $\mathcal{M}_{S} \to \mathcal{M}_{R}$ is separable in the sense of \cite{Nastasescu/alt:1989} if and only if  the extension $R \to S$ is separable (i.e., the multiplication map $S \otimes_R S \to S$ splits as an $S$--bimodule epimorphism). In general, if $(L,R)$ is an adjunction, then $R$ is a separable functor if and only if its counit is a split natural epimorphism (\cite[Theorem 1.2]{Rafael}).

\begin{lemma}
\label{lem: sep=>proj}Let $\left( L:\mathcal{B}\rightarrow \mathcal{A},R:%
\mathcal{A}\rightarrow \mathcal{B}\right) $ be an adjunction. If $R$ is separable, then $\left( L,\epsilon L\right) $ is relatively
projective as a right module functor on $\left( RL,R\epsilon L,\eta \right) $
\end{lemma}

\begin{proof}
By assumption, there is a natural transformation $\sigma :\mathrm{Id}_{%
\mathcal{A}}\rightarrow LR$ such that $\epsilon \circ \sigma =\mathrm{Id}_{%
\mathrm{Id}_{\mathcal{A}}}.$ Set $\gamma :=\sigma L.$ Then $\gamma $ is a
natural transformation such that $\epsilon L\circ \gamma =\mathrm{Id}_{L}$
and $LR\epsilon L\circ \gamma RL=\gamma \circ \epsilon L$. Then $\epsilon
L:\left( LRL,LR\epsilon L\right) \rightarrow \left( L,\epsilon L\right) $
splits as a morphism of right module functors.
\end{proof}

In the following result, part 3) may be compared, in its dual version, with
\cite[Proposition 3.16]{Mesab} and the results quoted therein.

\begin{proposition}\label{separable}
\label{pro: sep mon}Let $\left( L:\mathcal{B}\rightarrow \mathcal{A},R:%
\mathcal{A}\rightarrow \mathcal{B}\right) $ be an adjunction with unit $\eta
$ and counit $\epsilon $.

1) If $R$ is a separable functor then the comparison functor $R_{1}:\mathcal{%
A}\rightarrow \mathcal{B}_1$ is full and faithful.

2) Suppose that the comparison functor $R_{1}:\mathcal{A}\rightarrow {%
\mathcal{B}_1}$ has a left adjoint $L_{1}$. If $\left( L,\epsilon L\right) $
is relatively projective as a right module functor on $\left( RL,R\epsilon
L,\eta \right) ,$ then $L_{1}$ is full and faithful.

3) Suppose that the comparison functor $R_{1}:\mathcal{A}\rightarrow {%
\mathcal{B}_1}$ has a left adjoint $L_{1}$. If $R$ is a separable functor,
then $R$ is monadic.
\end{proposition}

\begin{proof}
1)\ By assumption there is a natural transformation $\sigma :\mathrm{Id}_{%
\mathcal{A}}\rightarrow LR$ such that $\epsilon \circ \sigma =\mathrm{Id}_{%
\mathrm{Id}_{\mathcal{A}}}$. Let $f, g : X \to Y$ morphisms in $\mathcal{A}$ such that $R_1 f = R_1 g$. Since $R = U R_1$, we get
\[
\sigma Y \circ f = LR f \circ \sigma X = LR g \circ \sigma X = \sigma Y \circ  g
\]
Now, $\sigma Y$ is a monomorphism, whence $f = g$. Thus, $R_1$ is faithful. 
To check that $R_1$ is full, consider a morphism $h : R_1 X \to R_1 Y$ in $\mathcal{B}_1$, and put $h' = \epsilon Y \circ Lh \circ \sigma X$. Since $h$ is a morphism of algebras, we get
\[
R h' = R \epsilon Y \circ RL h \circ R \sigma X \overset{h\in \mathcal{B}_1}{=} h \circ R \epsilon X \circ R \sigma X = h.
\]

2) By \cite[Proof of Theorem 1]{Beck}, since $L_{1}$ exists, there exists a
morphism $\pi $ such that%
\begin{equation}
LRLB\overset{L\mu }{\underset{\epsilon LB}{\rightrightarrows }}LB\overset{%
\pi }{\longrightarrow }L_{1}\left( B,\mu \right)  \label{form:coeqprim}
\end{equation}%
is a coequalizer for all $\left( B,\mu \right) \in \mathcal{B}_1$.
Moreover $L_{1}$ is full and faithful whenever%
\begin{equation*}
RLRLB\overset{RL\mu }{\underset{R\epsilon LB}{\rightrightarrows }}RLB\overset%
{R\pi }{\longrightarrow }RL_{1}\left( B,\mu \right)
\end{equation*}%
is a coequalizer too. By assumption there is a natural transformation $%
\gamma :L\rightarrow LRL$ such that $\epsilon L\circ \gamma =\mathrm{Id}_{L}$
and $LR\epsilon L\circ \gamma RL=\gamma \circ \epsilon L$. Clearly, we have $%
\epsilon LB\circ \gamma B=\mathrm{Id}_{LB}$. Moreover%
\begin{eqnarray*}
\left( L\mu \circ \gamma B\right) \circ L\mu &=&L\mu \circ \gamma B\circ
L\mu \overset{\text{nat }\gamma }{=}L\mu \circ LRL\mu \circ \gamma RLB \\
&=&L\mu \circ \left( LR\epsilon LB\circ \gamma RLB\right) \\
&=&L\mu \circ \left( \gamma B\circ \epsilon LB\right) \\
&=&\left( L\mu \circ \gamma B\right) \circ \epsilon LB
\end{eqnarray*}%
so that there is a unique morphism $p:L_{1}\left( B,\mu \right) \rightarrow
LB$ such that $p\circ \pi =L\mu \circ \gamma B.$ We have%
\begin{equation*}
\pi \circ p\circ \pi =\pi \circ L\mu \circ \gamma B=\pi \circ \epsilon
LB\circ \gamma B=\pi
\end{equation*}%
so that, since $\pi $ is an epimorphism, we get $\pi \circ p=\mathrm{Id}%
_{L_{1}\left( B,\mu \right) }$. We have so proved that (\ref{form:coeqprim})
is a contractible coequalizer. Thus it is preserved by any functor, in
particular by $R$. Thus $L_{1}$ is full and faithful too.

3)
It follows from 1), 2) and Lemma \ref{lem: sep=>proj} that both $L_{1}$ and $R_{1}$ are full and faithful.
\end{proof}

\section{Examples}\label{examples}

Let us fix a field $\Bbbk$. Vector spaces and algebras
are meant to be over $\Bbbk$.  From any vector space $V$ we can construct its tensor algebra $TV=\Bbbk \oplus V\oplus
V^{\otimes 2}\oplus \cdots $. In fact, this is the object part of a functor $T : Vect_\Bbbk \to Alg_\Bbbk$ from the category $Vect_\Bbbk$ of vector spaces to the category $Alg_\Bbbk$ of (associative and unital) algebras. By $\Omega : Alg_\Bbbk \to Vect_\Bbbk$ we denote the forgetful functor.

\subsection{Vector spaces and algebras}

\begin{example}
\label{ex: VecAlg} If $A$ is an algebra, and $V$ a vector space, then the universal property of $TV$ gives a bijection
\begin{equation}\label{universaltensor}
Alg_\Bbbk (TV, A) \cong Vect_\Bbbk (V, \Omega A),
\end{equation}
which is natural in both variables. In other words, the functor $T : Vect_\Bbbk \to Alg_\Bbbk$ is left adjoint to the forgetful functor $\Omega : Alg_\Bbbk \to Vect_\Bbbk$. It is very well-known that $\Omega$ is a monadic functor (cf. \cite[Proposition 4.6.2]{Borceux2}). Next, we check that $T$ is a comonadic functor.

In fact, given  $V\in Vect_\Bbbk$,  consider  the canonical
projection $\pi =\pi V:\Omega TV\rightarrow V$  on degree one. Let us check that it is natural in $V.$ Let $%
f:V\rightarrow V^{\prime }$ be a morphism in $Vect_\Bbbk$. For all $z\in
V^{\otimes n}$ with $n\neq 1$,%
\begin{equation*}
\left( \pi V^{\prime }\circ \Omega T f\right) \left( z\right) =\pi V^{\prime
}\left( f^{\otimes n}\left( z\right) \right) =0=\left( f\circ \pi V\right)
\left( z\right) .
\end{equation*}%
For $v\in V,$ we have%
\begin{equation*}
\left( \pi V^{\prime }\circ \Omega T f\right) \left( v\right) =\pi V^{\prime
}\left( f\left( v\right) \right) =f\left( v\right) =\left( f\circ \pi
V\right) \left( v\right) .
\end{equation*}%
so that $\pi V^{\prime }\circ \Omega T f=f\circ \pi V$ and $\pi V$ is natural in $V$%
. Moreover, we have $\pi V\circ i_{V}=\mathrm{Id}_{V}$, where $i_V : V \to \Omega T V$ is the canonical inclusion map
for every $V\in Vect_\Bbbk$. Since $i_V$ gives the unit of the adjunction at $V$, we can apply Rafael's Theorem \cite[Theorem 1.2]%
{Rafael}, to obtain that $T$ is a separable functor. By the dual version of
Proposition \ref{pro: sep mon}, in order to prove that $T$ is comonadic it
suffices to check that $T^{1}$ has a right adjoint. This follows by Beck's
Theorem \cite[Proof of Theorem 1]{Beck} as $Vect_\Bbbk$ has equalizers.
\end{example}

\subsection{Vector spaces and bialgebras}

\begin{example}
\label{ex: Vec-Bialg}
 Let $Bialg_\Bbbk$ be the category of bialgebras and
$
\Omega : Bialg_\Bbbk \rightarrow Vect_\Bbbk
$%
be the forgetful functor. By  $P : Bialg_\Bbbk \to Vect_\Bbbk$ we denote the functor that sends a bialgebra $A$ to its space $PA$ of primitive elements.  Obviously, $P$ is a subfunctor of $\Omega$, let  $j:P\rightarrow \Omega$ denote the inclusion natural transformation. We know that the tensor algebra $TV$ of a vector space $V$ is already a bialgebra. Therefore, the bijection \eqref{universaltensor} gives, by restriction, a bijection
\[
Bialg_\Bbbk (TV, A) \cong Vect_\Bbbk (V, PA)
\]
which, of course, is natural.

In this way, we see that $T$ is left adjoint to $P$. We will prove that $P$ has a\emph{\ }monadic decomposition of monadic length at most
$2.$ First we need to prove a technical result.
\end{example}

\begin{lemma}
\label{lem: coeq}Let $\left( L:\mathcal{B}\rightarrow \mathcal{A},R:\mathcal{%
A}\rightarrow \mathcal{B}\right) $ be an adjunction and let $\left( B,\mu
\right) \in \mathcal{B}_1$. Let $\zeta :B\rightarrow Z$ be a morphism
in $\mathcal{B}$. Then%
\begin{equation*}
\zeta \circ L\mu =\zeta \circ \epsilon LB\Leftrightarrow R\zeta \circ \eta
B\circ \mu =R\zeta .
\end{equation*}
\end{lemma}

\begin{proof}
Consider the canonical isomorphism $\Phi \left( X,Y\right) :\mathrm{Hom}_{%
\mathcal{A}}\left( LX,Y\right) \rightarrow \mathrm{Hom}_{\mathcal{B}}\left(
X,RY\right) $ defined by $\Phi \left( X,Y\right) f=Rf\circ \eta X.$ Then%
\begin{eqnarray*}
\zeta \circ L\mu &=&\zeta \circ \epsilon LB\Leftrightarrow \\
\Phi \left( RLB,Z\right) \left[ \zeta \circ L\mu \right] &=&\Phi \left(
RLB,Z\right) \left[ \zeta \circ \epsilon LB\right] \Leftrightarrow \\
R\left[ \zeta \circ L\mu \right] \circ \eta RLB &=&R\left[ \zeta \circ
\epsilon LB\right] \circ \eta RLB\Leftrightarrow \\
R\zeta \circ RL\mu \circ \eta RLB &=&R\zeta \circ R\epsilon LB\circ \eta
RLB\Leftrightarrow \\
R\zeta \circ \eta B\circ \mu &=&R\zeta .
\end{eqnarray*}
\end{proof}

\begin{theorem}\label{primitives}
\label{teo: classicLie}The functor $P$ has a\emph{\ }monadic decomposition
of monadic length at most $2$. Keep the notations of Definition \ref{def: comparable} (so, in particular, $\mathcal{B}_0 = Vect_\Bbbk$, $R_0 = P$,
and $L_0 = T$).
\begin{itemize}
\item[1)] The functor $L_{1}$ is given, for all $\left( V_{0},\mu
_{0}\right) \in \mathcal{B}_{1}$, by
\begin{equation*}
L_{1}\left( V_{0},\mu _{0}\right) =\frac{L_{0}V_{0}}{\left( \mathrm{Im}%
\left( \mathrm{Id}_{R_{0}L_{0}V_{0}}-\eta _{0}V_{0}\circ \mu _{0}\right)
\right) }.
\end{equation*}

\item[2)] The adjunction $\left( L_{1},R_{1}\right) $ is idempotent.

\item[4)] For all $V_{2}:=\left( \left( V_{0},\mu _{0}\right) ,\mu
_{1}\right) \in \mathcal{B}_{2},$ we have the following cases.

\begin{itemize}
\item $\mathrm{char}\Bbbk =0$. Then, for all $x,y\in V_{0}$ we have that $%
xy-yx\in R_{0}L_{0}V_{0}.$ Define a map $\left[ -,-\right] :V_{0}\otimes
V_{0}\rightarrow V_{0}$ by setting $\left[ x,y\right] :=\mu _{0}\left(
xy-yx\right) .$ Then $\left( V_{0},\left[ -,-\right] \right) $ is an
ordinary Lie algebra and $L_{2}V_{2}$ is the universal enveloping algebra%
\begin{equation*}
\mathfrak{U}V_{0}:=\frac{TV_{0}}{\left( xy-yx-\left[ x,y\right] \mid x,y\in
V_{0}\right) }.
\end{equation*}

\item $\mathrm{char}\Bbbk =p,$ a prime. Then, for all $x,y\in V_{0}$ we have
that $xy-yx,x^{p}\in R_{0}L_{0}V_{0}.$ Define two maps $\left[ -,-\right]
:V_{0}\otimes V_{0}\rightarrow V_{0}$ and $-^{\left[ p\right]
}:V_{0}\rightarrow V_{0}$ by setting $\left[ x,y\right] :=\mu _{0}\left(
xy-yx\right) $ and $x^{\left[ p\right] }:=\mu _{0}\left( x^{p}\right) .$
Then $\left( V_{0},\left[ -,-\right] ,-^{\left[ p\right] }\right) $ is a
restricted Lie algebra and $L_{2}V_{2}$ is the restricted enveloping algebra%
\begin{equation*}
\mathfrak{u}V_{0}:=\frac{TV_{0}}{\left( xy-yx-\left[ x,y\right] ,x^{p}-x^{%
\left[ p\right] }\mid x,y\in V_{0}\right) }.
\end{equation*}
\end{itemize}
\end{itemize}
\end{theorem}

\begin{proof}
Note that $\mathcal{A} = Bialg_\Bbbk$ has coequalizers (see e.g. \cite[page 1478]{Agore}).
Thus, using the notations of Definition \ref{def: comparable}, by Beck's
Theorem \cite[Proof of Theorem 1]{Beck}, the functors $L_{1}$ and $L_{2}$
exist. By construction, for every $V_{1}:=\left( V_{0},\mu
_{0}:R_{0}L_{0}V_{0}\rightarrow V_{0}\right) \in \mathcal{B}_{1}$ we have
that $L_{1}V_{1}$ is given by the coequalizer in $\mathcal{A}$ of the
diagram
\begin{equation*}
L_{0}R_{0}L_{0}V_{0}\overset{L_{0}\mu _{0}}{\underset{\epsilon _{0}L_{0}V_{0}%
}{\rightrightarrows }}L_{0}V_{0}.
\end{equation*}%
We want to compute explicitly this coequalizer. To this aim, we set%
\begin{equation*}
T_{1}V_{1}:=\frac{L_{0}V_{0}}{\left( S\right) },
\end{equation*}%
where $S:=\mathrm{Im}\left( \mathrm{Id}_{R_{0}L_{0}V_{0}}-\eta
_{0}V_{0}\circ \mu _{0}\right) ,$ and let us check it is a bialgebra. It is
enough to check that
\begin{eqnarray*}
\Delta _{L_{0}V_{0}}S &\subseteq &\left( S\right) \otimes
L_{0}V_{0}+L_{0}V_{0}\otimes \left( S\right) , \\
\varepsilon _{L_{0}V_{0}}S &=&0.
\end{eqnarray*}%
Both equalities follow trivially since $S\subseteq R_{0}L_{0}V_{0}=PTV.$
Hence $T_{1}V_{1}\in \mathcal{A}$. Let us check that
\begin{equation*}
L_{0}R_{0}L_{0}V_{0}\overset{L_{0}\mu _{0}}{\underset{\epsilon _{0}L_{0}V_{0}%
}{\rightrightarrows }}L_{0}V_{0}\overset{\pi }{\longrightarrow }T_{1}V_{1}
\end{equation*}%
is a coequalizer in $\mathcal{A}$, where $\pi $ is the canonical projection.
Let $\zeta :L_{0}V_{0}\rightarrow Z$ be a morphism in $\mathcal{A}$. By
Lemma \ref{lem: coeq}%
\begin{equation*}
\zeta \circ L_{0}\mu _{0}=\zeta \circ \epsilon _{0}L_{0}V_{0}\Leftrightarrow
R_{0}\zeta \circ \eta _{0}V_{0}\circ \mu _{0}=R_{0}\zeta \Leftrightarrow
\zeta \text{ vanishes on }S.
\end{equation*}%
Hence we can take $L_{1}V_{1}:=T_{1}V_{1}.$

We need to describe $%
L_{1}V_{1}$ in a different way for every $V_{1}:=\left( V_{0},\mu
_{0}\right) \in \mathcal{B}_{1}$. Note that $R_{0}L_{0}V_{0}=V_{0}\oplus
EV_{0}$ where $EV_{0}$ denotes the subspace of $\Omega L_{0}V_{0}$ spanned
by primitive elements of homogeneous degree greater than one. Let $%
x_{1}=\eta _{0}V_{0}:V_{0}\rightarrow R_{0}L_{0}V_{0}$ and $%
x_{2}:EV_{0}\rightarrow R_{0}L_{0}V_{0}$ be the canonical injections and set
$b:=\mu _{0}\circ x_{2}:EV_{0}\rightarrow V_{0}$. Let $c:V_{0}\otimes
V_{0}\rightarrow V_{0}\otimes V_{0}$ be the canonical flip. Then $b$ is a
bracket for the braided vector space $\left( V_{0},c\right) $ in the sense
of \cite[Definition 3.2]{Ardi-PrimGen}. We compute%
\begin{equation*}
\left( \mathrm{Id}_{R_{0}L_{0}V_{0}}-\eta _{0}V_{0}\circ \mu _{0}\right)
\circ x_{1}=x_{1}-\eta _{0}V_{0}\circ \mu _{0}\circ x_{1}=\eta
_{0}V_{0}-\eta _{0}V_{0}\circ \mu _{0}\circ \eta _{0}V_{0}=0
\end{equation*}%
so that%
\begin{equation*}
S=\mathrm{Im}\left( \mathrm{Id}_{R_{0}L_{0}V_{0}}-\eta _{0}V_{0}\circ \mu
_{0}\right) =\mathrm{Im}\left[ \left( \mathrm{Id}_{R_{0}L_{0}V_{0}}-\eta
_{0}V_{0}\circ \mu _{0}\right) \circ x_{2}\right] =\mathrm{Im}\left(
x_{2}-\eta _{0}V_{0}\circ b\right)
\end{equation*}%
and hence%
\begin{equation*}
L_{1}V_{1}=\frac{L_{0}V_{0}}{\left( S\right) }=\frac{L_{0}V_{0}}{\left(
\mathrm{Im}\left( x_{2}-\eta _{0}V_{0}\circ b\right) \right) }=\frac{%
L_{0}V_{0}}{\left( z-b\left( z\right) \mid z\in EV_{0}\right) }.
\end{equation*}%
Therefore $L_{1}V_{1}=U\left( V_{0},c,b\right) $ in the sense of \cite[%
Definition 3.5]{Ardi-PrimGen}.

Let now $V_{2}:=\left( V_{1},\mu _{1}\right) \in \mathcal{B}_{2}$. Then $%
V_{1}$ is of the form $\left( V_{0},\mu _{0}\right) $. By construction, the
unit of the adjunction is the unique map $\eta _{1}V_{1}:V_{1}\rightarrow
R_{1}L_{1}V_{1}$ such that%
\begin{equation*}
U_{0,1}\eta _{1}V_{1}=R_{0}\pi \circ \eta _{0}V_{0}.
\end{equation*}

 Consider the canonical map $i_{U}:V_{0}\rightarrow U\left( V_{0},c,b\right) $
i.e.%
\begin{equation*}
i_{U}=\Omega \pi \circ jL_{0}V_{0}\circ \eta _{0}V_{0}=jL_{1}V_{1}\circ
R_{0}\pi \circ \eta _{0}V_{0}=jL_{1}V_{1}\circ U_{0,1}\eta _{1}V_{1}
\end{equation*}%
so that $i_{U}$ corestricts to $U_{0,1}\eta _{1}V_{1}.$ Now
\begin{equation*}
U_{0,1}\mu _{1}\circ U_{0,1}\eta _{1}V_{1}=U_{0,1}\left( \mu _{1}\circ \eta
_{1}V_{1}\right) =\mathrm{Id}_{V_{0}}
\end{equation*}%
so that $U_{0,1}\eta _{1}V_{1}$ is injective. Therefore $i_{U}$ is
injective. This means that $\left( V_{0},c,b\right) $ is a braided Lie
algebra in the sense of \cite[Definition 4.1]{Ardi-PrimGen}. Let $\mathcal{S}
$ denote the class of braided vector spaces of combinatorial rank at most
one. Then $\left( V_{0},c\right) \in \mathcal{S}$

(see \cite[Example 6.10]{Ardi-AMM}, if $\mathrm{char}\left( \Bbbk \right) =0$%
, and \cite[Example 3.13]{Ardi-PBW}, if $\mathrm{char}\left( \Bbbk \right)
\neq 0$).

By \cite[Corollary 5.5]{Ardi-PrimGen}, we have that $U_{0,1}\eta _{1}V_{1}$
is an isomorphism. Since $U_{0,1}$ reflects isomorphism, we get that $\eta
_{1}V_{1}$ is an isomorphism. We have so proved that $\eta _{1}U_{1,2}$ is
an isomorphism. By Remark \ref{re:etaU}, we have that the adjunction $\left( L_{1},R_{1}\right)
$ is idempotent. By Proposition \ref{pro:idemmonad2}, the functor $R_{1}$
has a\emph{\ }monadic decomposition of monadic length at most $1$ so that $R$
has monadic decomposition of monadic length at most $2.$

We have observed that $\left( V_{0},c,b\right) $ is a braided Lie algebra in
the sense of \cite[Definition 4.1]{Ardi-PrimGen}.

The last part of the statement follows by \cite[Remark 6.4]{Ardi-PrimGen} in
case $\mathrm{char}\Bbbk =0$ and by the same argument as in \cite[Example
3.13]{Ardi-PBW} in case $\mathrm{char}\Bbbk =p$.
\end{proof}

\begin{remark}
In the setting of Theorem \ref{teo: classicLie}, $R=P:\mathcal{A}\rightarrow
\mathcal{B}$ has a\emph{\ }monadic decomposition of monadic length at most $2$.
Thus, by Theorem \ref{teo: descentFunct}, we have that%
\begin{equation*}
\mathrm{Im}R=\mathrm{Im}U_{0,2}.
\end{equation*}%
Note, since $\left( L_{1},R_{1}\right) $ is idempotent, we can apply
Proposition \ref{pro:idempotent} to get that an object in $\mathrm{Im}%
U_{0,2}$ is isomorphic to an object of the form $U_{0,2}\left( V_{1},\mu
_{1}\right) =U_{0,1}V_{1}$ for some $V_{1}\in \mathcal{B}_{1}$ such that $%
\eta _{1}V_{1}$ is an isomorphism.
\end{remark}

\begin{remark}
Let $\left( L,R\right) $ be the adjunction considered in \ref{ex: Vec-Bialg}%
. For a moment let $L^{\prime }$ denote the left adjoint $L$ of Example \ref%
{ex: VecAlg}. Let $W$ be the forgetful functor from the category of
bialgebras to the category of algebras. Then $W\circ L=L^{\prime }.$ Hence,
in view of \cite[Lemma 1.1]{Nastasescu/alt:1989}, from separability of $%
L^{\prime }$ we deduce separability of $L.$ Since $\mathcal{B}$ has all
equalizers, by the dual version of Beck's Theorem \cite[Proof of Theorem 1]%
{Beck}, we have that the comparison functor $L^{1}:\mathcal{B}\rightarrow {%
\mathcal{A}^1}$ has a right adjoint $R^{1}.$ Thus, by the dual version of
Proposition \ref{pro: sep mon}, we have that $L$ is comonadic.

Now, as observed in \cite[Theorem 2.3]{Agore}, in view of \cite[page 134]{Sw}%
, the functor $W$ has a right adjoint, say $\Gamma $. Explicitly $\Gamma A$
is the cofree bialgebra associated to $A$, for any algebra $A$. Now $\left(
WL,R\Gamma \right) $ is an adjunction as composition of adjunctions. Since $%
W\circ L=L^{\prime }$ and $\left( L^{\prime },R^{\prime }\right) $ is an
adjunction, we deduce that $R\Gamma $ is functorially isomorphic to $%
R^{\prime }$.
\end{remark}

\subsection{Pretorsion theories}

\begin{example}
Let $A$ be a ring and let $\mathcal{T}$ be a full subcategory of $Mod$-$A$
closed under submodules, quotients and direct sums i.e. $\mathcal{T}$ is an
hereditary pretorsion class. Let $t:Mod$-$A\rightarrow \mathcal{T}$ be the
associated left exact preradical (\cite[Corollary 1.8 page 138]{Stenstroem}%
). Then $R=t$ is a right adjoint of the inclusion functor $L=i:\mathcal{T}%
\rightarrow Mod$-$A$. Note that $RL=\mathrm{Id}_{\mathcal{T}}$ and $\eta =%
\mathrm{Id}_{\mathrm{Id}_{\mathcal{T}}}$ so that $L$ is full and faithful. Hence, $R$ has a\emph{\ }monadic decomposition of monadic
length\emph{\ }$0$. By Remark \ref{coalgebras}, the comparison functor  $C: \mathcal{T}\rightarrow ({Mod}\text{-}A)^1$ is a category equivalence.

As a particular example we consider the case when $A=C^{\ast }$ for some
coalgebra $C$ over a field $\Bbbk $ and $\mathcal{T}$ is the class of
rational right $C^{\ast }$-modules i.e. the image of the canonical functor $C
$-$CoMod\rightarrow Mod$-$C^{\ast }$.
\end{example}

\end{document}